\documentclass[a4paper, 10pt]{scrartcl}
\usepackage[twoside, bindingoffset=1cm, top=3cm, bottom=3.5cm, width=14cm]{geometry}

\makeatletter
\DeclareOldFontCommand{\rm}{\normalfont\rmfamily}{\mathrm}
\DeclareOldFontCommand{\bf}{\normalfont\bfseries}{\mathbf}
\DeclareOldFontCommand{\it}{\normalfont\itshape}{\mathit}
\makeatother

\usepackage[T1]{fontenc}
\usepackage{ucs}
\usepackage[utf8x]{inputenc}

\usepackage[reqno,intlimits]{mathtools}

\usepackage{amssymb, accents, amsthm, mathrsfs, dsfont, enumitem, cases}
\usepackage{enumitem}
\newlist{conditions}{enumerate}{1}
\setlist[conditions]{label=(\alph*),ref=(\alph*)}
\makeatletter
\newcommand{\mylabel}[2]{#2\def\@currentlabel{#2}\label{#1}}
\makeatother
\usepackage{stackrel} 

\usepackage{color}

\usepackage[bookmarks, colorlinks, citecolor=blue]{hyperref}

\usepackage{graphicx, float}
\usepackage{ragged2e}
\usepackage[rightcaption]{sidecap}
\usepackage{wrapfig}
\usepackage[labelfont=bf, format=plain, font=footnotesize]{caption}

\usepackage{pdfpages, scrpage2}
\automark{chapter}

\newtheorem*{thm}{Theorem}
\newtheorem{lem}{Lemma}[section]
\newtheorem{cor}[lem]{Corollary}
\newtheorem{ass}[lem]{Assumption}
\newtheorem{defn}[lem]{Definition}
\newtheorem{remark}[lem]{Remark}

\expandafter\let\expandafter\oldproof\csname\string\proof\endcsname
\let\oldendproof\endproof
\renewenvironment{proof}[1][\proofname]{%
  \oldproof[{#1}]%
}{\oldendproof}

\newcommand{\Lw}{\mathcal{L}^w}

\newcommand{\del}{\partial}
\newcommand{\supp}{\text{supp}}

\newcommand{\N}{\mathbb{N}}
\newcommand{\R}{\mathbb{R}}
\newcommand{\Z}{\mathbb{Z}}
\newcommand{\E}{\mathbb{E}}
\renewcommand{\vec}{\boldsymbol}
\newcommand{\eps}{\epsilon}
\renewcommand{\Pr}{\mathbb{P}}
\newcommand{\Pas}{$\Pr$-a.s.\ }
\newcommand{\edges}{\mathfrak{E}}
\newcommand{\D}{\mathscr{D}}
\newcommand{\Dn}{\D^{(n)}}
\newcommand{\I}{\mathscr{I}}
\newcommand{\In}{\I^{(n)}}
\newcommand{\Bln}{\mathscr{B}^{(n)}_l}

\let\originalleft\left
\let\originalright\right
\renewcommand{\left}{\mathopen{}\mathclose\bgroup\originalleft}
\renewcommand{\right}{\aftergroup\egroup\originalright}

\title{\large Eigenvector localization in the heavy-tailed random conductance model}
  \date{}
  \author{\small Franziska Flegel\\\small Weierstrass Institute Berlin}

\begin{document}
\numberwithin{equation}{section}
\numberwithin{figure}{section}

\rohead{Eigenvector localization in the RCM}
\rehead{F.\ Flegel}
\lohead{F.\ Flegel}
\lehead{Eigenvector localization in the RCM}
\setheadsepline{.5pt}

\maketitle

\vspace{-1cm}
\begin{abstract}
  We generalize our former localization result about the principal Dirichlet eigenvector of the i.i.d.\ heavy-tailed random conductance Laplacian to the first $k$ eigenvectors.
  We overcome the complication that the higher eigenvectors have fluctuating signs by invoking the Bauer-Fike theorem to show that the $k$th eigenvector is close to the principal eigenvector of an auxiliary spectral problem.
\end{abstract}

\section{Introduction}
Let us consider the random conductance Laplacian $\Lw$ acting on real-valued functions $f\in\nobreak\ell^2(\Z^d)$ as
\begin{align}
  (\Lw f)(x) =\sum_{y\colon |x-y|_1 =1} w_{xy}(f(y) - f(x)) \qquad (x\in\Z^d)
  \label{equ:DefGenerator}
\end{align}
with positive independent and identically distributed random conductances $w_{xy}$.
As usual, we further assume that the operator $\Lw$ is self-adjoint, i.e.\ $w_{xy}=w_{yx}$.
Our goal is to describe the almost-sure behavior of the solution to the spectral problem
\begin{align}
\begin{aligned}
  -\Lw \psi &= \lambda \psi &&\text{on }B_n=[-n,n]^2 \cap \Z^d\, ,\\
  \psi &= 0 &&\text{else.}
  \label{equ:SpectralProblem}
\end{aligned}
\end{align}
as the box size $n$ tends to infinity.
This means that we are interested in the Dirichlet eigenfunctions and eigenvalues of the operator $-\Lw$ in the box $B_n$ with zero Dirichlet conditions.

In the recent paper \cite{Flegel2016}, we have shown that if $\gamma := \sup \{ q\geq 0\colon \E [w^{-q}] <\infty \} < 1/4$ and certain regularity assumptions apply, then the principal Dirichlet eigenvector $\psi_1^{(n)}$ of Problem \eqref{equ:SpectralProblem} concentrates in a single site as $n$ tends to infinity.
To be more precise, let $\pi_z = \sum_{x\colon x\sim z} w_{xz}$ be the local speed measure, i.e., the inverse mean waiting time of the random walk generated by $\Lw$.
Then the principal Dirichlet eigenvector $\psi_1^{(n)}$ approaches the $\delta$-function in the site $z_{(1,n)}$ that minimizes the local speed measure $\pi$ over the box $B_n$.
Furthermore, the principal Dirichlet eigenvalue $\lambda_1^{(n)}$ is asymptotically equivalent to the minimum $\pi_{1,B_n}=\min_{z\in B_n} \pi_z$.

If, on the other hand, $\gamma > 1/4$, then the authors of \cite{FHS2017} have proved that the top of the Dirichlet spectrum of $\Lw$ homogenizes.
The spectrum of the random conductance Laplacian thus displays a dichotomy between a localized and a homgenized phase.

In the present paper we generalize our findings for $\gamma<1/4$ to the first $k$ Dirichlet eigenvectors and eigenvalues.
More precisely, we show that the $k$th Dirichlet eigenvector $\psi_k^{(n)}$ concentrates in the site that attains the $k$th minimum of $\pi$.
Consequently, the $k$th Dirichlet eigenvalue $\lambda_k^{(n)}$ is asymptotically equivalent to the $k$th minimum of $\pi$.
If the conductances vary regularly at zero with positive index, then despite the dependence structure of the random field $\left\{\pi_x\right\}_{x\in\Z^d}$, this $k$th minimum converges weakly as if $\left\{\pi_x\right\}_{x\in\Z^d}$ was an independent field, see the proof of Corollary \ref{cor:ConvInLaw}.
It follows that, in this case, the properly rescaled $k$th eigenvalue $\lambda_k^{(n)}$ converges in distribution to a non-degenerate random variable.
This relates to a similar result in dimension $d=1$, see \cite[Theorem 2.5(i)]{Faggionato2012}.

Note that the only reason why we have not generalized our findings to the first $k$ eigenvectors in \cite{Flegel2016}, is that in \cite[Lemma 5.6]{Flegel2016} we rely on the property that the principal Dirichlet eigenvector does not change its sign, according to the Perron-Frobenius theorem.
This is no longer true for the higher order eigenvectors.
To overcome this difficulty, we now approximate the first $k$ eigenvectors to \eqref{equ:SpectralProblem} by auxiliary principal eigenvectors using the Bauer-Fike theorem, see Lemma \ref{lem:BauerFike}.

Our results for the random conductance Laplacian compare well to similar results of the random Schr\"odinger operator $\Delta + \xi$ with random potential $\xi\colon \Z^d \to \R$, see \cite{Biskup2016} and \cite[Ch.\ 6]{Astrauskas2016}.
To keep the present paper as short as possible, we refer the reader to our first article \cite{Flegel2016} for more heuristics and references.
However, we kept the present paper mostly self-contained.

\subsection*{Model and main objects}
\label{subsec:model}
We consider the lattice with vertex set $\mathbb{Z}^d$ ($d\geq 2$) and edge set $\edges_d = \{ \{x,y\}: x,y \in \mathbb{Z}^d, |x-y|_1 = 1\}$.
If two sites $x,y\in \mathbb{Z}^d$ are neighbors according to $\edges_d$, we also write $x\sim y$.
To each edge $e\in \edges_d$ we assign a positive random variable $w_e$.
In analogy to a $d$-dimensional resistor network, we call these random weights $w_e$ \emph{conductances}.
We take $(\Omega,\mathcal{F})=\bigl( (0,\infty)^{\edges_d}, \mathcal{B} \left( (0,\infty) \right)^{\otimes\edges_d} \bigr)$ as the underlying measurable space and assume that an environment $\boldsymbol{w} = (w_e)_{ e\in\edges_d}\in \Omega$ is a family of i.i.d.\ positive random variables with law $\Pr$. We denote the expectation with respect to $\Pr$ by $\E$.

If $e$ is the edge between the sites $x,y\in \mathbb{Z}^d$, we also write $w_{xy}$ or $w_{x,y}$ instead of $w_e$.
Note that by definition of the edge set $\edges_d$, the edges are undirected, whence $w_{xy} = w_{yx}$.
If we want to refer to an arbitrary copy of the conductances in general, we simply write $w$, i.e., for a set $A\in\mathcal{B} \left( (0,\infty) \right)$, the expression $\Pr [w\in A]$ equals $\Pr [w_{e}\in A]$ for an arbitrary edge $e$.

We call
\begin{align}
  F\colon [0,\infty) \to [0,1]\colon u\mapsto \Pr [w\leq u]
\end{align}
the distribution function of the conductances.

For an arbitrary $k\in\N$, our goal is to study the behavior of the first $k$ Dirichlet eigenvalues $\lambda_1^{(n)}\leq \ldots \leq \lambda_k^{(n)}$ and eigenvectors $\psi_1^{(n)}, \ldots, \psi_k^{(n)}$ of the sign-inverted generator $-\mathcal{L}_{\boldsymbol{w}}$ in the ball
\begin{align}
  B_n := \left\{ x\in\mathbb{Z}^d\colon |x|_\infty \leq n\right\} = [-n,n]^d \cap \Z^d
\end{align}
with zero Dirichlet conditions at the boundary.

For a subset $A\subset \Z^d$ we define the function space
\begin{align}
  \ell^2 (A) := \left\{ f\colon \Z^d \to \R \text{ such that } \supp\, f \subseteq A \text{ and } \sum_{x\in A} f(x)^2 < \infty \right\} \subset \ell^2 (\Z^d)\, ,
  \label{equ:Defell2A}
\end{align}
where we let ``$\supp\, f$'' denote the support of the function $f$.
Accordingly, for functions $f_1, f_2\in\ell^2 (\Z^d)$ we define the scalar product
\begin{align*}
  \langle f_1, f_2 \rangle_{\ell^2 (A)} = \sum_{x\in A} f_1 (x) f_2 (x)\, .
\end{align*}

For a real-valued function $f\in \ell^2 (\Z^d)$ let us define the Dirichlet energy $\mathcal{E}^{\boldsymbol{w}} (f)$ with respect to the operator $-\mathcal{L}_{\boldsymbol{w}}$ by
\begin{align}
  \label{equ:DefDirichletEnergy}
  \mathcal{E}^{\boldsymbol{w}} (f) = \langle f, -\mathcal{L}_{\boldsymbol{w}} f \rangle_{\ell^2 (\Z^d)}\, .
\end{align}
Then, according to the Courant-Fischer theorem, the $k$th Dirichlet eigenvalue is given by the variational formula
\begin{align}
	\lambda_k^{(n)} = \inf_{\mathcal{M} \leq \ell^2 (B_n),\atop \text{dim}\, \mathcal{M} =k} \sup_{f\in \mathcal{M},\atop \| f \|_{2} = 1} \mathcal{E}^{\boldsymbol{w}} (f)
		\label{equ:Variational}
\end{align}
where $\mathcal{M} \leq \ell^2 (B_n)$ means that $\mathcal M$ is a linear subspace of $\ell^2 (B_n)$.
Note that $\lambda_k^{(n)} = \mathcal{E}^{\boldsymbol{w}} \bigl(\psi_k^{(n)}\bigr)$.

\begin{defn}[Local speed measure and its order statistics]
\label{def:Pi}
  We define the local speed measure $\pi$ by
  \begin{align}
    \pi_z = \sum_{x\colon x\sim z} w_{xz} \qquad(z\in\Z^d)
  \end{align}
  and we label the order statistics of the set $\left\{ \pi_z \right\}_{z\in B_n}$ by
  \begin{align}
    \pi_{1,B_n} \leq \pi_{2,B_n} \leq \ldots \leq \pi_{|B_n|,B_n}\, .
  \end{align}
  Furthermore, for $k,n\in\N$ let $z_{(k,n)}$ be the site where $\pi$ attains its $k$th minimum over $B_n$, i.e., $\pi_{z_{(k,n)}} = \pi_{k,B_n}$.
\end{defn}
\begin{remark}
\label{rem:ContF}
  If $F$ is continuous, then $\pi_{1,B_n} < \pi_{2,B_n} < \ldots < \pi_{|B_n|,B_n}$ \Pas and therefore the minimizers $z_{(k,n)}$ are \Pas unique.
\end{remark}

\section{Main result}
In what follows we let
  \begin{align}
    \label{def:ThmFuncG}
    g:[0,\infty)\to [0,\infty)\colon u\mapsto \sup\,\left\{ s\geq 0 \colon F(s) = u^{-1/2}\right\}\, .
  \end{align}

\begin{ass}
  \label{ass:main}
  Let $F$ be continuous and vary regularly at zero with index $\gamma\in [0,1/4)$.
  Assume that there exists $a^\ast>0$ such that $F(ab) \geq bF(a)$ for all $a\leq a^\ast$ and all $0\leq b\leq1$.
  In the case where $\gamma=0$, we assume additionally that there exists $\eps_1\in (0,1)$ such that the product $n^{2+\eps_1} g \left( n \right)$ converges monotonically to zero as $n$ grows to infinity.
\end{ass}

\begin{remark}
\label{rem:RegVarEps1}
  In the case where $\gamma>0$, it follows that $(1/F(1/s))^2$ varies regularly at infinity with index $2\gamma$.
Further, $(1/F(1/s))^2$ diverges as $s\to\infty$.
It follows by virtue of \cite[Prop. 0.8(v)]{Resnick1987} that $1/g(u) = \inf\,\left\{ s\geq 0\colon (1/F(1/s))^2 = u \right\}$ varies regularly at infinity with index $1/(2\gamma)$ and thus $g$ varies regularly at infinity with index $-1/(2\gamma)$.
Since in addition $\gamma<1/4$, there exists $\eps_1\in (0,1)$ such that $-1/(2\gamma) < -(2+\eps_1)$.
\end{remark}

\begin{thm}
  \label{MainThm}
  Let $k\in \N$.
  If Assumption \ref{ass:main} holds, then the $k$th Dirichlet eigenvalue $\lambda_k^{(n)}$ with zero Dirichlet conditions outside the box $B_n$ fulfills
  \begin{align}
    \Pr \left[ \lim_{n\to\infty} \frac{\lambda_k^{(n)}}{\pi_{k,B_n}} = 1 \right] = 1
    \label{equ:thm:EigValQual}
  \end{align}
  and the mass of the $k$th Dirichlet eigenvector $\psi_k^{(n)}$ asymptotically concentrates in the site $z_{(k,n)}$.
  More precisely, if $\eps_1>0$ is as in Assumption \ref{ass:main} or Remark \ref{rem:RegVarEps1}, then \Pas for $n$ large enough
  \begin{align}
      1 - n^{-\eps/8}\leq \frac{\lambda_k^{(n)}}{\pi_{k,B_n}} \leq 1\qquad\text{for all } \eps<\eps_1
      \label{equ:thm:EigVal}
  \end{align}
  and
  \begin{align}
    \psi_k^{(n)} \left(z_{(k,n)}\right) \geq \sqrt{1 - n^{-\eps/4}}\qquad\text{for all } \eps<\eps_1\, .
    \label{equ:thm:EigVec}
  \end{align} 
\end{thm}
We prove this theorem in Section \ref{sec:ProofMainThm}.

Similar to \cite[Corollary 1.11]{Flegel2016}, we can now infer the weak convergence of the eigenvalues.
Let $F_\pi$ be the distribution function of the random variable $\pi$, i.e., the distribution function of the sum of $2d$ independent copies of the conductance $w$.
Note that since $F$ is continuous, $F_\pi$ is continuous as well.
As in \cite[(1.18)]{Flegel2016}, we define
\begin{align}
 \label{equ:Defh}
 h\colon (0,\infty) \to (0,\infty)\colon u \mapsto \inf\,\left\{ s\colon \frac{1}{F_\pi (1/s)} = u \right\}\, .
\end{align}
Let $F$ vary regularly at zero with index $\gamma>0$.
Then by virtue of \cite[Lemma 5.8]{Flegel2016}, it follows that $F_\pi$ varies regularly at zero with index $2d\gamma$.
It thus follows by virtue of \cite[Proposition 0.8(v)]{Resnick1987} that $h$ varies regularly at infinity with index $1/(2d\gamma)$.
Therefore there exists a function $L^\ast$ that varies slowly at infinity such that
\begin{align}
\label{def:Last}
 h(|B_n|) = n^{\frac{1}{2\gamma}} L^\ast (n)\, .
\end{align}

\begin{cor}
\label{cor:ConvInLaw}
   Assume that $F$ fulfills Assumption \ref{ass:main} with $\gamma>0$ and let $L^\ast$ be as in \eqref{def:Last}.
   Let $k\in\N$. Then as $n$ tends to infinity, the product $L^\ast (n) n^{\frac{1}{2\gamma}} \lambda_k^{(n)}$ converges in distribution to a non-degenerate random variable.
   More precisely,
   \begin{align}
      \lim_{n\to \infty}\Pr \left[ L^\ast (n) n^{\frac{1}{2\gamma}} \lambda_k^{(n)} > \zeta\right] = \exp\left( -\zeta^{2d\gamma} \right)\sum_{j=0}^{k-1}\frac{\zeta^{2d\gamma j}}{j!}\qquad\text{for all } \zeta\in [0,\infty)\, .
   \end{align}
\end{cor}
This corollary extends \cite[Corollary 1.11]{Flegel2016} to general $k\in\N$.
We prove it at the end of Section \ref{sec:EigVals}.

\section{Auxiliary spectral problems}
\begin{defn}[Auxiliary lattice and Laplacian]
  We define the set
  \begin{align}
    \Bln = B_n \backslash \left\{ z_{(1,n)}, \ldots, z_{(l-1,n)} \right\}
  \end{align}
  and abbreviate the operator $\Lw$ with zero Dirichlet conditions outside $\Bln$ as $\Lw_{(l,n)}$, i.e., we define
  \begin{align}
    \Lw_{(l,n)} := \mathds{1}_{\Bln} \,\Lw\, \mathds{1}_{\Bln}\, ,
  \end{align}
  where the operator $\mathds{1}_{\Bln}$ is the identity on $\Bln$ and zero otherwise.
\end{defn}
Since the operator $-\Lw$ is self-adjoint, the operator $-\Lw_{(l,n)}$ is self-adjoint as well.
This justifies the next definition.
\begin{defn}[Auxiliary eigenvectors and values]
We define the eigenvalues of the operator $-\Lw_{(l,n)}$ restricted to $\ell^2\left( \Bln \right)$ by
  \begin{align}
    \mu_{l,1}^{(n)} \leq \mu_{l,2}^{(n)} \leq \ldots \leq \mu_{l,|\Bln|}^{(n)}
  \end{align}
  and its eigenvectors by
  \begin{align}
    \phi_{l,1}^{(n)}, \phi_{l,2}^{(n)}, \ldots, \phi_{l,|\Bln|}^{(n)} \in \ell^2\left( \Bln \right) \qquad \text{with}\quad \left\langle\phi_{l,i}^{(n)}, \phi_{l,j}^{(n)}\right\rangle = \delta_{ij}\, .
  \end{align}
\end{defn}
Note that $\mathscr{B}_1^{(n)} = B_n$ and thus $\mu_{1,k}^{(n)} = \lambda_k^{(n)}$ and $\phi_{1,k}^{(n)} = \psi_k^{(n)}$.
Moreover the variational formula for the auxiliary eigenvalues reads
\begin{align}
	\mu_{l,m}^{(n)} = \inf_{\mathcal{M} \leq \ell^2 (\Bln),\atop \text{dim}\, \mathcal{M} =m} \sup_{f\in \mathcal{M},\atop \| f \|_{2} = 1} \mathcal{E}^{\boldsymbol{w}} (f)\, .
		\label{equ:VariationalMu}
\end{align}

\begin{remark}[Perron-Frobenius]
 \label{rem:PF}
  For a given box $B_n$ the operator $\Lw_{(l,n)}$ can be written as a $(|B_n|-l+1) \times (|B_n|-l+1)$-matrix with non-negative entries everywhere except on the diagonal.
  Since the matrix is finite-dimensional, we can add a multiple of the identity to obtain a non-negative primitive matrix without changing the matrix' spectrum.
  By the Perron-Frobenius theorem (see e.g.\ \cite[Ch.\ 1]{SenetaNN}) it follows that its principal eigenvalue $-\mu_{l,1}^{(n)}$ is simple and we can assume without loss of generality that its principal eigenvector is positive, which implies that $\phi_{l,1}^{(n)}$ is nonnegative.
\end{remark}

\begin{lem}
  \label{lem:TrivialUpperBound}
  For any $l\in\N$ and $m\in\{1,\ldots,|B_n|-l+1\}$ the eigenvalue $\mu_{l,m}^{(n)}$ is bounded from above by
  \begin{align}
    \label{equ:lem1:EigVal:UpperBound}
    \mu_{l,m}^{(n)} \leq \pi_{l+m-1, B_n}\, .
  \end{align}
\end{lem}
\begin{proof}
  We choose
  \begin{align*}
    \mathcal{M} = \text{span}\, \left\{ \delta_{z_{(l,n)}}, \delta_{z_{(l+1,n)}}, \ldots, \delta_{z_{(l+m-1,n)}} \right\}
  \end{align*}
  and insert it as a test space into the variational formula \eqref{equ:VariationalMu}.
\end{proof}

\subsection{Principal eigenvectors}
The following lemma is the analogue of \cite[Lemma 5.6]{Flegel2016}, where we need the Perron-Frobenius property.
\begin{lem}
  \label{lem:uniquemax}
   Let $k\in\N$ and let $y,z\in B_n\cap\mathscr{B}_k^{(n)}$ with $\pi_z < \pi_y$ and $y\nsim z$.
   Assume that $\phi_{k,1}^{(n)}$ is nonnegative.
   Further, define $m_y = 2 \max_{x\colon x\sim y} \phi_{k,1}^{(n)}(x)$.
   Then the mass $\phi_{k,1}^{(n)} (y)$ is bounded from above by
   \begin{align}
	\phi_{k,1}^{(n)}(y) \leq \frac{m_y}{1 - \frac{\pi_z}{\pi_y}}\, .
	\end{align}
\end{lem}
The proof of this lemma is analogous to the proof of \cite[Lemma 5.6]{Flegel2016} and therefore we omit it here.

For the convenience of the reader, we now repeat some definitions from \cite{Flegel2016}.
For a function $g:(0,\infty)\to (0,\infty)$ and $n\in\N$ we define a percolation environment $\tilde{\boldsymbol{w}}_{g(n)}$ by setting
\begin{align}
	\tilde{w}_{g(n)} (e) := w_e \mathds{1}_{\{ w_e > g(n)\} } \qquad (e\in \edges_d)\, .
\end{align}
Thus, edges with conductance less than or equal to $g(n)$ are considered to be closed and all others keep their original conductance.
With this terminology we can now define the following clusters.
\begin{defn}\label{def:InDn}
  For a fixed function $g$ and a fixed $\eps>0$, let $\Dn$ be the unique infinite open cluster of the environment $\tilde{\boldsymbol{w}}_{g(n^{1-\eps})}$ and let $\In = B_n\backslash \Dn$ be its set of holes in $B_n$.
\end{defn}
\begin{defn}\label{def:sparse}
  We call a set $\mathscr{I}\subset\Z^d$ {\bf sparse} if the set $\mathscr{I}$ does not contain any neighboring sites.
  Further, a set $\mathscr{I}\subset\Z^d$ is {\bf $\mathbf{b}$-sparse} if for any $z\in\Z^d$ the box $B_b(z) := \left\{ x\in\Z^d \colon |x-z|_\infty \leq b \right\}\subset \Z^d$ contains at most one site of the set $\mathscr{I}$.
\end{defn}
\begin{remark}
  \label{rem:b1b2sparse}
  Let $b_1<b_2$ be natural numbers.
  If a set $\mathscr{I}\subset\Z^d$ is $b_2$-sparse, it is also $b_1$-sparse and sparse.
\end{remark}

Let us collect some facts that we already know about the cluster $\Dn$ and the set $\In$ from \cite{Flegel2016}.

\begin{remark}
  \label{rem:TechnicalStuff}
Let us recall that in Assumption \ref{ass:main} we assume that one of the two following cases occurs:
$\gamma\in (0,1/4)$ or $\gamma=0$ and there exists $\eps_1\in (0,1)$ such that the product $n^{2+\eps_1} g(n)$ converges monotonically to zero as $n$ grows to infinity.
In the case where $\gamma\in (0,1/4)$, we define $\eps_1$ as in Remark \ref{rem:RegVarEps1}.

In both cases we define $\Dn$ and $\In$ as in Definition \ref{def:InDn} with $\eps= \eps_2:=\frac{7\eps_1}{8(2+\eps_1)}$.
By virtue of \cite[Lemma 5.4]{Flegel2016} and Remark \ref{rem:b1b2sparse} we know that for any fixed $b\in\N$ the set $\In$ is $b$-sparse and therefore sparse \Pas for $n$ large enough in the sense of Definition \ref{def:sparse}.
Moreover, \cite[Lemma 5.4]{Flegel2016} implies that for any $k\in\N$ we have \Pas for $n$ large enough $z_{(1,n)}, \ldots, z_{(k+1,n)}\in\In$ and thus
\Pas for $n$ large enough there is no pair of neighbors among the the sites $z_{(1,n)}, \ldots, z_{(k+1,n)}$.
Since $F$ is continuous, the sites $z_{(1,n)}, \ldots, z_{(k+1,n)}$ are \Pas unique.
\end{remark}

The next lemma about the principal Dirichlet eigenvector $\phi_{k,1}^{(n)}$ of the auxiliary operator $-\Lw_{(k,n)}$ is very similar to \cite[Lemma 5.5]{Flegel2016}.
Indeed, we can nearly copy the proof since the deleted sites $z_{(1,n)}, \ldots, z_{(k-1,n)}$ are in $\In$, see Remark \ref{rem:TechnicalStuff}.

\begin{lem}
  \label{lem:Loc:MassDn}
  Let the function $g$ be as in \eqref{def:ThmFuncG}.
  Assume that there exists $\eps_1\in (0,1)$ such that one of the two cases occurs: $g$ varies regularly at infinity with index $\rho<-(2+\eps_1)$ or the product $n^{2+\eps_1} g(n)$ converges monotonically to zero as $n$ grows to infinity.
  Further, let $\eps= \eps_2 := \frac{7\eps_1}{8(2+\eps_1)}$ and $\Dn$ be as in Definition \ref{def:InDn}.
  Then \Pas for $n$ large enough
  \begin{align}
     \bigl\| \phi_{k,1}^{(n)} \bigr\|^2_{\ell^2(\Dn)} \leq n^{-\eps_1/2}\, .
  \end{align}
\end{lem}
\begin{proof}
  The proof follows the lines of the proof of \cite[Lemma 5.5]{Flegel2016} until \emph{right before} (5.8).
  Here, we then apply Lemma \ref{lem:TrivialUpperBound} to infer that
  \begin{align*}
    \pi_{k, B_n}\geq \mu_{k,1}^{(n)} = \mathcal{E}^{\vec{w}} \left( \phi_{k,1}^{(n)} \right)\, .
  \end{align*}
  Moreover, by virtue of \cite[Lemma 2.6]{Flegel2016} there exists $c_1<\infty$ such that \Pas for $n$ large enough
  \begin{align*}
  c_1 g(n^{1-\eps_3}) \geq \pi_{k, B_n}
\end{align*}
  with $\eps_3 = \eps_1 (8(2+\eps_1))^{-1}$.
  The rest of the proof follows again the lines of the proof of \cite[Lemma 5.5]{Flegel2016}.
\end{proof}

From Lemma \ref{lem:Loc:MassDn} to localization in a single site, the main two ingredients are Lemma \ref{lem:uniquemax} and the following result about the order statistics of $\left\{ \pi_x \right\}_{x\in B_n}$.

\begin{lem}[{\cite[Lemma 5.10]{Flegel2016}}]
\label{lem:QuotOrder}
  Let Assumption \ref{ass:main} be true and let $\varepsilon>0$ and $k\in\N$.
  Then \Pas for $n$ large enough
  \begin{align}
    1-\frac{\pi_{k,B_n}}{\pi_{k+1,B_n}} > n^{-\varepsilon}\, .
  \end{align}
\end{lem}

The next lemma therefore follows.

\begin{lem}
  \label{lem:MuPhi1}
  Let $k\in\N$.
  Under Assumption \ref{ass:main}, it follows that \Pas for $n$ large enough
  \begin{align}
    \phi_{k,1}^{(n)} \left(z_{(k,n)}\right) \geq \sqrt{1 - n^{-\eps_1/4}}\, .
    \label{equ:lem1:EigVec}
  \end{align}
  This implies that \Pas for $n$ large enough
  \begin{align}
    \label{equ:lem1:EigVal:LowerBound}
    \mu_{k,1}^{(n)} \geq \left( 1 - 2n^{-\eps_1/8} \right) \pi_{k,B_n}\, .
  \end{align}
\end{lem}
\begin{proof}
In view of Remark \ref{rem:TechnicalStuff}, Lemma \ref{lem:uniquemax} and the extreme value result Lemma \ref{lem:QuotOrder}, the proof of \eqref{equ:lem1:EigVec} is completely analogous to the proof of \cite[Theorem 1.8]{Flegel2016} and thus we omit it here.
For \eqref{equ:lem1:EigVal:LowerBound} we observe that since $\mu_{k,1}^{(n)} = \langle \phi_{k,1}^{(n)}, \Lw \phi_{k,1}^{(n)} \rangle$ it follows that \Pas for $n$ large enough
\begin{align*}
  \mu_{k,1}^{(n)} &\;\geq\; \sum_{x\colon x\sim z_{(k,n)}} w_{xz_{(k,n)}} \left( \phi_{k,1}^{(n)} (z_{(k,n)}) - \phi_{k,1}^{(n)} (x) \right)^2
  \;\geq\; \left( n^{-\eps_1/8} - \sqrt{1 - n^{-\eps_1/4}} \right)^2 \pi_{z_{(k,n)}}\, .
\end{align*}
\end{proof}

\subsection{Orthogonality of eigenvectors}
The next very simple ingredient of our proof is due to the orthogonality of the eigenvectors.
\begin{lem}
\label{lem:orth}
  Let $\varepsilon>0$, let $j,l,m,n\in\N$ with $j<m$ and let $\phi_{l,j}^{(n)} (z) \geq \sqrt{1 - n^{-\varepsilon/4}}$.
  \begin{align}
    \left| \phi_{l,m}^{(n)} (z)\right| &\leq n^{-\varepsilon/8}\, .
  \end{align}
\end{lem}
\begin{proof}
  For $n=1$ the claim is immediate.
  For $n\geq 2$ we observe that since the eigenvectors $\phi_{l,j}^{(n)}$ and $\phi_{l,m}^{(n)}$ are orthogonal to each other, it follows that
  \begin{align*}
    \phi_{l,m}^{(n)} (z) &= -\frac{\sum_{x\neq z} \phi_{l,j}^{(n)}(x) \phi_{l,m}^{(n)} (x)}{\phi_{l,j}^{(n)}(z)}\, .
  \end{align*}
  By the Cauchy-Schwarz inequality it follows that for $n$ greater than one
  \begin{align*}
    \left( \phi_{l,m}^{(n)} (z)\right)^2 &\;\leq\; \frac{\left( \sum_{x\neq z}\left(\phi_{l,j}^{(n)}(x)\right)^2 \right) \left( 1 - \left(\phi_{l,m}^{(n)}(z)\right)^2\right)}{\left( \phi_{l,j}^{(n)} (z)\right)^2}
    \leq \frac{n^{-\varepsilon/4}}{1 - n^{-\varepsilon/4}}\left( 1 - \left(\phi_{l,m}^{(n)}(z)\right)^2\right)
  \end{align*}
  where we have also used that the assumption implies that $\sum_{x\neq z}\left(\phi_{l,j}^{(n)}(x)\right)^2 \leq  n^{-\varepsilon/4}$.
  The claim follows.
\end{proof}

\subsection{Higher eigenvalues and -vectors}
We establish the connection to the original eigenvalues and -vectors via the Bauer-Fike theorem \cite{Bauer1960}, which we cite below from \cite[Lemma 11.2]{JKO1994}.
\begin{lem}[{\cite[Lemma 11.2]{JKO1994}}]
  \label{lem:BauerFike}
  Let $A\colon H\to H$ be a linear self-adjoint compact operator in a Hilbert space $H$.
  Let $\mu\in\R$, and let $u\in H$ be such that $\| u \|_H =1$ and
  \begin{align}
    \| Au - \mu u\|_H \leq \alpha\, ,\qquad \alpha>0\, .
  \end{align}
  Then there exists an eigenvalue $\mu_i$ of the operator $A$ such that
  \begin{align}
    |\mu_i - \mu|\leq \alpha\, .
  \end{align}
  Moreover, for any $\beta>\alpha$, there exists a vector $\overline u$ such that
  \begin{align}
    \| u-\overline u \|_H \leq 2\alpha\beta^{-1}\, , \qquad \| \overline u \|_H = 1
  \end{align}
  and $\overline u$ is a linear combination of the eigenvectors of operator $A$ corresponding to the eigenvalues from the interval $[\mu-\beta, \mu+\beta]$.
\end{lem}

Here comes the first application of Lemma \ref{lem:BauerFike}.
\begin{lem}
  \label{lem:MuClose1}
  Let $l\in\N$ and $m\in\{1,\ldots,|B_n|-l+1\}$.
  Under Assumption \ref{ass:main} there exists $i\in\{1,\ldots, |B_n|-l+1\}$ such that
  \begin{align}
    \label{equ:MuClose1}
    \left|\mu_{l,i}^{(n)} - \mu_{l+m,1}^{(n)} \right| \leq n^{-\eps_1/4}\cdot \pi_{l+m-1,B_n} \, .
  \end{align}
\end{lem}
\begin{proof}
   We aim to apply Lemma \ref{lem:BauerFike} with the operator $A = -\Lw_{(l,n)}$, the Hilbert space $H = \ell^2 (\Bln)$, the value $\mu = \mu_{l+m,1}$ and the vector $u = \phi_{l+m,1}^{(n)}$.
   First, we note that $\| \phi_{l+m,1}^{(n)} \|_{\ell^2 (\Bln)} = 1$.
   Next, we recall that $\phi_{l+m,1}^{(n)}$ is an eigenvector of the operator $-\Lw_{(l+m,n)}$ to the eigenvalue $\mu_{l+m,1}^{(n)}$ and therefore
  \begin{align*}
    \left\| \Lw_{(l,n)} \phi_{l+m,1}^{(n)} \mspace{-5mu}+ \mu_{l+m,1}^{(n)} \phi_{l+m,1}^{(n)}\right\|^2_{\ell^2 (\Bln)}\mspace{-3mu}
    &= \mspace{-13mu}\sum_{z\in\Bln\backslash\mathscr{B}^{(n)}_{l+m}}\mspace{-11mu} \Bigl(\Lw_{(l,n)} \phi_{l+m,1}^{(n)} (z)  + \mu_{l+m,1}^{(n)} \phi_{l+m,1}^{(n)} (z)\Bigr)^2\, ,
  \end{align*}
  where all other summands vanish.
  Note that $\Bln\backslash\mathscr{B}^{(n)}_{l+m} = \left\{ z_{(l,n)}, \ldots z_{(l+m-1,n)} \right\}$ and by definition we have $\phi_{l+m,1}^{(n)} (z)=0$ for all $z\in\left\{ z_{(l,n)}, \ldots z_{(l+m-1,n)} \right\}$.
  It follows that for all $z\in\left\{ z_{(l,n)}, \ldots, z_{(l+m-1,n)} \right\}$ we have
  \begin{align*}
    \Lw_{(l,n)} \phi_{l+m,1}^{(n)} (z) \;=\; \sum_{x\colon x\sim z} w_{xz} \Bigl( \phi_{l+m,1}^{(n)} (x) - \phi_{l+m,1}^{(n)} (z)\Bigr)
    \;=\; \sum_{x\colon x\sim z} w_{xz} \phi_{l+m,1}^{(n)} (x)\, .
  \end{align*}
  Since $\pi_{l+m-1,B_n}\geq \pi_{l+m-2,B_n} \geq \ldots \geq \pi_{l,B_n}$, it follows that
  \begin{align*}
  \left\| \Lw_{(l,n)} \phi_{l+m,1}^{(n)} + \mu_{l+m,1}^{(n)} \phi_{l+m,1}^{(n)}\right\|^2_{\ell^2 (\Bln)}
  &\leq \pi^2_{l+m-1,B_n}\sum_{z\in\Bln\backslash\mathscr{B}^{(n)}_{l+m}} \max_{x\colon x\sim z} \Bigl(  \phi_{l+m,1}^{(n)} (x)\Bigr)^2\, .
  \end{align*}
  Since by virtue of Remark \ref{rem:TechnicalStuff} the sites $z_{(1,n)}, \ldots, z_{(l+m-1,n)}$ are in $\In$ and are neither neighbors nor do they share a common neighbor \Pas for $n$ large enough, it follows that \Pas for $n$ large enough
  \begin{align*}
    \sum_{z\in\Bln\backslash\mathscr{B}^{(n)}_{l+m}} \max_{x\colon x\sim z} \Bigl(  \phi_{l+m,1}^{(n)} (x)\Bigr)^2 \leq \sum_{x\in \Dn} \left( \phi_{l+m,1}^{(n)} (x)\right)^2 \leq n^{-\eps_1/2}\, ,
  \end{align*}
  where the last bound is due to Lemma \ref{lem:Loc:MassDn}.
  The claim follows by virtue of Lemma \ref{lem:BauerFike}.
\end{proof}

Here comes the second application of Lemma \ref{lem:BauerFike}.
\begin{lem}
  \label{lem:MuClose2}
  Let $\varepsilon>0$, $l,m\in\N$.
  If Assumption \ref{ass:main} holds and \Pas for $n$ large enough
  \begin{align}
    \phi_{l,j}^{(n)} \left(z_{(l+j-1,n)}\right) \geq \sqrt{1 - n^{-\varepsilon/4}}\qquad \text{for all } 1\leq j\leq m\, ,
    \label{equ:lemMuClose2:Cond:EigVec}
  \end{align}
  then \Pas for $n$ large enough there exists $j\in\{1,\ldots, |B_n|-l-m+1\}$ such that
  \begin{align}
  \label{equ:MuClose2}
    \left|\mu_{l,m+1}^{(n)} - \mu_{l+m,j}^{(n)} \right| \leq \pi_{l+m-1, B_n}\, \sqrt{\frac{mn^{-\varepsilon /4}}{1 - mn^{-\varepsilon /4}}}\, .
  \end{align}
\end{lem}
\begin{proof}
  We aim to apply Lemma \ref{lem:BauerFike} with the operator $A = -\Lw_{(l+m,n)}$, the Hilbert space $H = \ell^2 (\mathscr{B}^{(n)}_{l+m})$, the value $\mu = \mu_{l,m+1}^{(n)}$ and the vector $u = \phi_{l,m+1}^{(n)}/\|\phi_{l,m+1}^{(n)} \|_{\ell^2 (\mathscr{B}^{(n)}_{l+m})}$.
  First, we note that by definition $\| u \|_{\ell^2 (\mathscr{B}^{(n)}_{l+m})} = 1$ and \Pas for $n$ large enough
  \begin{align}
    \|\phi_{l,m+1}^{(n)} \|^2_{\ell^2 (\mathscr{B}^{(n)}_{l+m})} = 1 - \mspace{-20mu}\sum_{z\in \Bln\backslash \mathscr{B}^{(n)}_{l+m}} \left( \phi_{l,m+1}^{(n)} (z) \right)^2 \geq 1 - mn^{-\varepsilon/4}
    \label{equ:EstMassPhiM1}
  \end{align}
  by virtue of Condition \eqref{equ:lemMuClose2:Cond:EigVec} and Lemma \ref{lem:orth}.
  
  Next, as we show in detail in \eqref{equ:App1}, we can estimate
  \begin{align}
    &\left\| \Lw_{(l+m,n)} \phi_{l,m+1}^{(n)} + \mu_{l,m+1}^{(n)} \phi_{l,m+1}^{(n)}\right\|^2_{\ell^2 (\mathscr{B}^{(n)}_{l+m})}\nonumber\\
    &\mspace{200mu}\leq \max_{z\in \Bln\backslash\mathscr{B}^{(n)}_{l+m}} \left( \phi_{l,m+1}^{(n)}(z)\right)^2 \sum_{x\in B_n}\mspace{-3mu} \Biggl( \sum_{z\colon z\sim x\atop z\in \Bln\backslash\mathscr{B}^{(n)}_{l+m}} \mspace{-15mu}w_{xz} \Biggr)^2\, .
    \label{equ:lem:proof:MuClose2}
  \end{align}
  Since by virtue of Remark \ref{rem:TechnicalStuff} we have \Pas for $n$ large enough
  \begin{align*}
   \Bln\backslash\mathscr{B}^{(n)}_{l+m} = \left\{ z_{l,n}, \ldots, z_{l+m-1,n} \right\} \subset \In
  \end{align*}
  and $\In$ is 1-sparse, it follows that on the RHS of \eqref{equ:lem:proof:MuClose2} for each $x\in B_n$ the sum over $\left\{ z\in \Bln\backslash\mathscr{B}^{(n)}_{l+m}\colon z\sim x \right\}$ contains at most one summand.
  Therefore \Pas for $n$ large enough we can pull the square into the inner sum.
  Then we rearrange both sums and use that for all $z$ we have $\sum_{x\colon x\sim z} w_{xz}^2 \leq \pi_z^2$ to infer that \Pas for $n$ large enough
  \begin{align*}
   \left\| \Lw_{(l+m,n)} \phi_{l,m+1}^{(n)} + \mu_{l,m+1}^{(n)} \phi_{l,m+1}^{(n)}\right\|^2_{\ell^2 (\mathscr{B}^{(n)}_{l+m})}&\leq \max_{z\in \Bln\backslash\mathscr{B}^{(n)}_{l+m}} \left( \phi_{l,m+1}^{(n)}(z)\right)^2 \sum_{z\in \Bln\backslash\mathscr{B}^{(n)}_{l+m}} \pi_{z}^2\, .
  \end{align*}
  By virtue of Lemma \ref{lem:orth} and Assumption \eqref{equ:lemMuClose2:Cond:EigVec}, for all $z\in \{ z_{(l,n)},\ldots, z_{(l+m-1,n)} \}$ we know that \Pas for $n$ large enough
  \begin{align*}
    \left| \phi_{l,m+1}^{(n)}(z)\right|  \leq n^{-\varepsilon/8}\, .
  \end{align*}
  Furthermore, $\sum_{z\in \Bln\backslash\mathscr{B}^{(n)}_{l+m}} \pi_{z}^2\leq m \pi_{l+m-1,B_n}^2$.
  It follows that \Pas for $n$ large enough
  \begin{align*}
  \left\| \Lw_{(l+m,n)} \phi_{l,m+1}^{(n)} - \mu_{l,m+1}^{(n)} \phi_{l,m+1}^{(n)}\right\|^2_{\ell^2 (\mathscr{B}^{(n)}_{l+m})} \leq mn^{-\varepsilon/4} \pi_{l+m-1,B_n}^2\, .
  \end{align*}
  
  Together with \eqref{equ:EstMassPhiM1} it follows that \Pas for $n$ large enough
  \begin{align}
  \left\| \Lw_{(l+m,n)} u - \mu_{l,m+1}^{(n)} u\right\|^2_{\ell^2 (\mathscr{B}^{(n)}_{l+m})} \leq \frac{mn^{-\varepsilon/4}}{1 - mn^{-\varepsilon/4}}\, \pi_{l+m-1,B_n}^2\, .
  \label{equ:CondBauerFike2}
  \end{align}
  and therefore the claim follows by virtue of Lemma \ref{lem:BauerFike}.
\end{proof}

Both Lemmas \ref{lem:MuClose1} and \ref{lem:MuClose2} imply the following lemma.

\begin{lem}
  \label{lem:InductiveStep}
  Let $\eps\in(0,\eps_1)$ and $l,m\in\N$.
  If Assumption \ref{ass:main} holds and \Pas for $n$ large enough
  \begin{align}
    \phi_{l,j}^{(n)} \left(z_{(l+j-1,n)}\right) \geq \sqrt{1 - n^{-\eps/4}}\qquad \text{for all } 1\leq j\leq m \, ,
    \label{equ:lem2:Cond:EigVec}
  \end{align}
  then
  \begin{align}
    \label{equ:lem2:EigVal}
    \mu_{l,m+1}^{(n)} \geq \left( 1 - (2+\sqrt{m})n^{-\eps/ 8} \right)\pi_{l+m, B_n}\, .    
  \end{align}
\end{lem}
\begin{proof}  
  Let us first assume that $\mu_{l,m+1}^{(n)} \leq \mu_{l+m,1}^{(n)}$.
  Due to Assumption \eqref{equ:lem2:Cond:EigVec} we can apply Lemma \ref{lem:MuClose2}.
  Because of the ordering $\mu_{l+m,1}^{(n)} \leq \mu_{l+m,2}^{(n)} \leq \ldots, $ it follows that Relation \eqref{equ:MuClose2} holds with $j=1$ and $\varepsilon=\eps$.
  On the other hand, if $\mu_{l,m+1}^{(n)} > \mu_{l+m,1}^{(n)}$, then \eqref{equ:MuClose1} holds with an index $i\leq m+1$.
  Let us now argue why \eqref{equ:MuClose1} holds with exactly $i= m+1$ \Pas for $n$ large enough.
  We assume the contrary, i.e., that $i\leq m$ infinitely often as $n$ tends to infinity.
  Then \eqref{equ:MuClose1} together with \eqref{equ:lem1:EigVal:LowerBound} implies that 
  \begin{align*}
    \mu_{l,i}^{(n)} \geq \mu_{l+m,1}^{(n)} - n^{-\eps_1/4} \pi_{l+m-1,B_n} \geq \left(1 - 2n^{-\eps_1/8} - n^{-\eps_1/4}\right) \pi_{l+m,B_n}
  \end{align*}
  Note that \eqref{equ:lem1:EigVal:UpperBound} implies that $\mu_{l,i}^{(n)} \leq \pi_{l+i-1,B_n}$, which we assumed to be less than or equal to $\pi_{l+m-1,B_n}$ infinitely often as $n$ tends to infinity.
  Thus
  \begin{align*}
    \frac{\pi_{l+m-1,B_n}}{\pi_{l+m,B_n}} \geq 1 - 3n^{-\eps_1/8}
  \end{align*}
  infinitely often as $n$ tends to infinity.
  This is a contradiction to Lemma \ref{lem:QuotOrder}.
  
  Thus, since $\eps<\eps_1$, it follows regardless of whether $\mu_{l,m+1}^{(n)} \leq \mu_{l+m,1}^{(n)}$ or $\mu_{l,m+1}^{(n)} > \mu_{l+m,1}^{(n)}$ that \Pas for $n$ large enough
  \begin{align}
  \label{equ:MuClose}
    \left|\mu_{l,m+1}^{(n)} - \mu_{l+m,1}^{(n)} \right| \leq \sqrt{\frac{mn^{-\eps/4}}{1 - mn^{-\eps/4}}}\, \pi_{l+m-1,B_n} \leq \sqrt{m} n^{-\eps/8}\cdot \pi_{l+m, B_n}\, .
  \end{align}
  Therefore \Pas for $n$ large enough $\mu_{l,m+1}^{(n)}$ is bounded from below by
  \begin{align}
  \label{equ:MuCloseLower}
      \mu_{l,m+1}^{(n)} &\;\geq\; \mu_{l+m,1}^{(n)} - \sqrt{m} n^{-\eps/8}\cdot \pi_{l+m, B_n}\;
      \stackrel{\text{\eqref{equ:lem1:EigVal:LowerBound}}}\geq\; \left( 1 - (2+\sqrt{m})n^{-\eps/8}\right)\pi_{l+m, B_n}\, .
  \end{align}
\end{proof}

Now we have the ingredients to prove the main theorem by induction.

\section{Proof of the main theorem}
\label{sec:ProofMainThm}
By virtue of Lemma \ref{lem:TrivialUpperBound}, we already know that
\begin{align*}
  \lambda_k^{(n)} \leq \pi_{k,B_n}\qquad \text{for all }k\in\N\, .
\end{align*}
In what follows, we further prove \eqref{equ:thm:EigVec} and that \Pas for $n$ large enough
\begin{align*}
 \lambda_k^{(n)} \geq \left( 1 - n^{-\eps/8}\right)\pi_{k,B_n}\qquad\text{for all }\eps<\eps_1\, .
\end{align*}

We prove the claim by induction over $k$.
\paragraph{Base case: $\boldsymbol{k=1}$.}
\Pas for $n$ large enough we have
  \begin{align}
  \psi_1^{(n)} \left(z_{(1,n)}\right)^2 \geq  1 - n^{-\eps_1/4}\, ,
  \label{equ:Psi1zn}
  \end{align}
by virtue of \cite[Theorem 1.8]{Flegel2016} and
  \begin{align}
    \lambda_1^{(n)} \geq \left( 1 - 2n^{-\eps_1/8}\right)\pi_{1,B_n} > \left( 1 - n^{-\eps/8}\right)\pi_{1,B_n}\qquad\text{for all }\eps<\eps_1
    \label{equ:PrincipalEigValMinPi}
  \end{align}
by virtue of \cite[Equation (5.30)]{Flegel2016}.

\paragraph{Inductive step: $\boldsymbol{(k-1) \rightsquigarrow k}$.}
Suppose that the claims \eqref{equ:thm:EigVal} and \eqref{equ:thm:EigVec} hold for some $k-1\in\N$.
We now show that this implies that the claims also hold for $k$ instead of $k-1$.

For \eqref{equ:thm:EigVal} this already follows by Lemma \ref{lem:InductiveStep} with $l=1$ and $m=k-1$.
Note that here Condition \eqref{equ:lem2:Cond:EigVec} holds for all $\eps<\eps_1$ and therefore \eqref{equ:lem2:EigVal} holds even without the multiplicative constants.
For \eqref{equ:thm:EigVec} we apply the second part of Lemma \ref{lem:BauerFike}:
Let $0<\delta<\eps_1/16$ and
\begin{align}
  \beta_k^{(n)} = 2\sqrt{k-1}\, n^{-\delta} \pi_{k,B_n}\, .
\end{align}
Since $\pi_{k-1,B_n} \leq \pi_{k,B_n}$, it follows that $\beta_k^{(n)} > \alpha_k^{(n)}$ with
\begin{align*}
  \alpha_k^{(n)} := \sqrt{k-1}\,  n^{-\eps_1/8} \pi_{k-1,B_n} \, .
\end{align*}
Therefore Lemma \ref{lem:BauerFike} and \eqref{equ:CondBauerFike2} with $l=1$ and $m=k-1$ imply that there exists a function $\overline u\colon \Z^d \to \R$ such that
\begin{align}
  \left\| \psi_k^{(n)} - \overline u\right\|_{\ell^2 (B_n)} \leq \frac{2\sqrt{k-1}\, n^{-\eps_1/8} \pi_{k-1,B_n}}{\beta_k^{(n)}}
\end{align}
where $\overline u$ is a linear combination of the eigenvectors $\left\{ \phi_{k,j} \right\}_{j\geq 1}$ corresponding to the eigenvalues from the interval $\left[\lambda_k^{(n)}-\beta_k^{(n)}, \lambda_k^{(n)}+\beta_k^{(n)}\right]$ of the operator $-\Lw_{(k,n)}$.
We now show that \Pas for $n$ large enough $\overline u = \phi_{k,1}^{(n)}$, i.e., that \Pas for $n$ large enough
\begin{align}
  \text{spec}\,\Lw_{(k,n)} \cap \left[\lambda_k^{(n)}-\beta_k^{(n)}, \lambda_k^{(n)}+\beta_k^{(n)}\right] = \left\{ \mu_{k,1}^{(n)} \right\}\, .
  \label{equ:SingleEigVal}
\end{align}
It suffices to show that \Pas for $n$ large enough $\mu_{k,2}^{(n)} > \lambda_k^{(n)}+\beta_k^{(n)}$.
We note that Lemma \ref{lem:TrivialUpperBound} implies that
\begin{align}
  \lambda_{k}^{(n)} +\beta_k^{(n)} \leq \left( 1 + 2\sqrt{k-1}\, n^{-\delta} \right)\pi_{k,B_n}\, .
\end{align}
By virtue of Lemma \ref{lem:QuotOrder} we have \Pas for $n$ large enough $\frac{\pi_{k,B_n}}{\pi_{k+1, B_n}}< 1 - 2\sqrt{k-1}\, n^{-\delta}$, whence it follows that \Pas for $n$ large enough
\begin{align*}
  \lambda_{k}^{(n)} +\beta_k^{(n)} < \left( 1 - 4(k-1) n^{-2\delta} \right)\pi_{k+1, B_n} \leq \mu_{k,2}^{(n)}\, ,
\end{align*}
where the last inequality follows since by the inductive assumption the relation \eqref{equ:lem2:Cond:EigVec} holds for all $\eps<\eps_1$ and therefore \eqref{equ:lem2:EigVal} holds for all $\eps<\eps_1$ with $l=k$ and $m=1$.
Therefore \eqref{equ:SingleEigVal} is true.

It follows that for any $0<\delta<\eps_1/16$ we have \Pas for $n$ large enough
\begin{align*}
  \left| \psi_k^{(n)}(z_{(k,n)}) - \phi_{k,1}^{(n)}(z_{(k,n)})\right| \;\leq\; \frac{n^{\delta-\eps_1/8} \pi_{k-1,B_n}}{\pi_{k,B_n}} \;<\; n^{\delta-\eps_1/8}\, .
\end{align*}
By virtue of Lemma \ref{lem:MuPhi1}, we already know that \Pas for $n$ large enough $\left| \phi_{k,1}^{(n)} \left(z_{(k,n)}\right)\right| \geq \sqrt{1 - n^{-\eps_1/4}}$.
It follows that
\begin{align*}
  \left( \psi_k^{(n)}(z_{(k,n)}) \right)^2
  \;\geq\; 1 - n^{-\eps_1/4} + n^{2\delta - \eps_1/4} - 2n^{\delta - \eps_1/8} \;\geq\; 1 - 2n^{\delta - \eps_1/8}\, .
\end{align*}
The claim follows since we can choose $\delta$ arbitrarily small.

\section{Asymptotics of the eigenvalues}
\label{sec:EigVals}
The proof of Corollary \ref{cor:ConvInLaw} extends the proof of \cite[Corollary 1.11]{Flegel2016}, which uses the ideas of \cite{Watson1954}.
To keep the present paper self-contained, we repeat the initial definitions and statements.
We define
\begin{align*}
  a_n \;:=\; \left( n^{\frac{1}{2\gamma}} L^\ast (n)\right)^{-1}\;=\;\frac{1}{h (|B_n|)} = \sup\, \left\{ t\colon F_\pi (t) = |B_n|^{-1} \right\}
\end{align*}
with $h$ as in \eqref{equ:Defh} and $L^\ast (n)$ as in \eqref{def:Last}.
Then $|B_n| = \left( \Pr \left[ \pi_0 \leq a_n\right] \right)^{-1}$ and therefore
\begin{align}
  \lim_{n\to\infty} |B_n| \Pr \left[ \pi_0 \leq a_n \zeta \right] = \lim_{n\to\infty}\frac{F_\pi \left( a_n \zeta \right)}{F_\pi \left( a_n \right)} \;=\;\zeta^{2d\gamma} \quad\text{for all }\zeta\geq 0
  \label{equ:WatsonCond1}
\end{align}
since $a_n\to 0$ as $n\to\infty$ and $F_\pi$ varies regularly at zero with index $2d\gamma$.
We further note that if $\vec{e}_1\in \Z^d$ is a neighbor of the origin, then $\Pr \left[ \left\{ \pi_{0} \leq a_n \zeta\right\} \cap \left\{ \pi_{\vec{e}_1} \leq a_n \zeta\right\} \right] \leq F(a_n\zeta)^{4d-1}$ since for the event $\left\{ \pi_{0} \leq a_n \zeta\right\} \cap \left\{ \pi_{\vec{e}_1} \leq a_n \zeta\right\}$ at least $4d-1$ independent conductances $w$ have to be smaller than or equal to $a_n \zeta$.
Since $F$ varies regularly at zero with index $\gamma$, it follows that   
\begin{align}
  |B_n|\, \Pr \left[ \left\{ \pi_{0} \leq a_n \zeta\right\} \cap \left\{ \pi_{\vec{e}_1} \leq a_n \zeta\right\} \right]  \to 0\qquad\text{as }n\to \infty\, .
  \label{equ:WatsonCond2}
\end{align}

We start with the auxiliary Lemma \ref{lem:AuxConvInProb}, for which we need some further definitions.  
For a set $A\subset \Z^d$ we define $CC(A)$ as the set of connected components of $A$.
Furthermore, we define the outer site boundary of the set $A$ as
\begin{align}
  \del A := \left\{ z\in\Z^d\backslash A \colon \exists x\in A \text{ with } x\sim z\right\}\, .
\end{align}
For the natural numbers $q\leq m$ we further define the number
\begin{align}
 C_{m,q}^{(n)} (A)  := \bigl| \left\{ M\subset B_n\backslash (A\cap \del A) \colon |M|=m,|CC(M)|=q\right\}\bigr|\, .
 \label{equ:LeadOrderCmq}
\end{align}
\begin{remark}
\label{rem:LeadOrderCmq}
Note that if we fix a $k\in\N$, then as $n$ tends to infinity we have $C_{m,m}^{(n)} (A_n)= |B_n|^m / m! + O\left( |B_n|^{m-1} \right)$ for all sequences of subsets $A_n\in B_n$ with the constraint $|A_n|=k-1$.
Moreover, for $q\leq m-1$ there exists a constant $c_q<\infty$ such that for all $n\in\N$ and all sequences of subsets $A_n\subset B_n$ with $|A_n|=k-1$, we have $C_{m,q}^{(n)} (A_n) < c_q|B_n|^{q}$.
Note that this $c_q$ is independent of the specific choice of $A_n$.
\end{remark}
\begin{lem}
  \label{lem:AuxConvInProb}
  For any fixed $k,l\in\N$ the relations \eqref{equ:WatsonCond1} and \eqref{equ:WatsonCond2} imply that
  \begin{align}
    \lim_{n\to\infty} \sup_{A_n\subset B_n, \atop |A_n|=k-1} \sum_{m=1}^{l}\sum_{q=1}^{m-1}\sum_{M\subset B_n\backslash (A_n\cap \del A_n),\atop {|M|=m,\atop |CC(M)|=q}} \Pr \left[ \bigcap_{x\in M}\left\{ \pi_{x} \leq a_n\zeta\right\} \right] = 0 \text{ for all }\zeta\geq 0\, .
  \end{align}
\end{lem}
\begin{proof}
We are summing over sets $M$ with the constraint $|CC(M)|=q<m=|M|$.
This means that here all the sets $M$ contain at least one connected component $\mathscr{C}$ with a neighboring pair of sites, i.e., $\Pr \left[ \bigcap_{x\in \mathscr{C}}\left\{ \pi_{x} \leq a_n\zeta\right\} \right] \leq \Pr \left[ \left\{ \pi_{0} \leq a_n\zeta\right\} \cap \left\{ \pi_{\vec{e}_1} \leq a_n\zeta\right\}\right]$.
Since $\pi_x$ and $\pi_y$ are independent if the sites $x$ and $y$ are in two different connected components of $M$, it follows that
\begin{align*}
  &\sum_{m=1}^{l}\sum_{q=1}^{m-1}\sum_{M\subset B_n\backslash (A_n\cap \del A_n),\atop {|M|=m,\atop |CC(M)|=q}} \Pr \left[ \bigcap_{x\in M}\left\{ \pi_{x} \leq a_n\zeta\right\} \right]\\
  &\mspace{100mu}\leq \sum_{m=1}^{l}\sum_{q=1}^{m-1} C_{m,q}^{(n)} (A_n) \Pr \left[ \pi_{0} \leq a_n\zeta\right]^{q-1}\Pr \left[ \left\{ \pi_{0} \leq a_n\zeta\right\} \cap \left\{ \pi_{\vec{e}_1} \leq a_n\zeta\right\}\right]\, .
\end{align*}
By Remark \ref{rem:LeadOrderCmq} there exists a constant $c_q<\infty$ such that $C_{m,q}^{(n)} (A_n)\leq c_q |B_n|^q$ for all sequences of subsets $A_n \subset B_n$ with the constraint that $|A_n| = k-1$.
Therefore the claim follows by \eqref{equ:WatsonCond1} and \eqref{equ:WatsonCond2}.
\end{proof}

\begin{proof}[Proof of Corollary \ref{cor:ConvInLaw}]
Because of the main theorem it remains to show that
\begin{align}
  \lim_{n\to\infty} \Pr \left[ \pi_{k,B_n} > \frac{\zeta}{n^{\frac{1}{2\gamma}}L^\ast(n)} \right] = \exp\left( -\zeta^{2d\gamma} \right)\sum_{j=0}^{k-1}\frac{\zeta^{2d\gamma j}}{j!}\quad\text{for all }\zeta\geq 0\, .
  \label{equ:CorRemainsPi}
\end{align}
The proof extends the proof of \cite[Corollary 1.11]{Flegel2016}, where we have already shown that
\begin{align}
  \lim_{n\to\infty} \Pr \left[ \min_{x\in B_n}\pi_x > a_n \zeta \right] = \exp\left( -\zeta^{2d\gamma} \right)\qquad \text{for all }\zeta\geq 0
  \label{equ:Cor0}
\end{align}
by extending the ideas of \cite{Watson1954} from $d=1$ to $d\geq 2$.
We will use \eqref{equ:Cor0} for the inductive base case $k=1$.

In what follows all the statements hold for all $\zeta\geq 0$.
For the inductive step we consider
\begin{align*}
  \Pr \left[ \pi_{k,B_n} > a_n \zeta \right]
    &= \Pr \left[ \pi_{k-1,B_n} > a_n \zeta \right] + \Pr \left[ \left\{ \pi_{k,B_n} > a_n \zeta \right\} \cap \left\{ \pi_{k-1,B_n} \leq a_n \zeta \right\}\right]\, .
\end{align*}
Let us now assume that the claim \eqref{equ:CorRemainsPi} holds for some $k-1$.
It follows that it remains to show that
\begin{align*}
  \lim_{n\to\infty} \Pr \left[ \pi_{k,B_n} > a_n \zeta,\, \pi_{k-1,B_n} \leq a_n \zeta \right] = \frac{\zeta^{2(k-1)d\gamma}}{(k-1)!} \exp\left( -\zeta^{2d\gamma} \right)\, .
\end{align*}
Let us start with the decomposition
\begin{align}
  \Pr \left[ \pi_{k,B_n} > a_n \zeta,\, \pi_{k-1,B_n} \leq a_n \zeta \right]
  &= \sum_{A\subset B_n,\atop |A|=k-1} \Pr \left[ \bigcap_{x\in A}\left\{ \pi_x \leq a_n \zeta \right\} \cap \mspace{-13mu}\bigcap_{y\in B_n\backslash A} \mspace{-13mu}\left\{ \pi_y > a_n \zeta \right\} \right]\nonumber\\
  &\mspace{-150mu}= \sum_{A\subset B_n,\atop |A|=k-1} \Pr \left[ \bigcap_{x\in A}\left\{ \pi_x \leq a_n \zeta \right\} \cap \mspace{-13mu}\bigcap_{y\in B_n\backslash (A\cap \del A)} \mspace{-13mu}\left\{ \pi_y > a_n \zeta \right\} \right]\nonumber\\
  &\mspace{-10mu}- \sum_{A\subset B_n,\atop |A|=k-1} \Pr \left[ \bigcap_{x\in A}\left\{ \pi_x \leq a_n \zeta \right\} \cap \left( \bigcup_{y\in \del A} \left\{ \pi_y \leq a_n \zeta \right\} \right) \right]\, .
  \label{equ:Cor1}
\end{align}
Let us argue that the second term on the above RHS converges to zero.
We observe that
\begin{align*}
  &\sum_{A\subset B_n,\atop |A|=k-1} \Pr \left[ \bigcap_{x\in A}\left\{ \pi_x \leq a_n \zeta \right\} \cap \left( \bigcup_{y\in \del A} \left\{ y \leq a_n \zeta \right\} \right) \right]\\
  &\mspace{80mu}\leq \sum_{A\subset B_n,\atop |A|=k-1} \sum_{y\in \del A} \Pr \left[\left\{ \pi_y \leq a_n \zeta \right\}  \cap \bigcap_{x\in A}\left\{ \pi_x \leq a_n \zeta \right\} \right]
  \leq \mspace{-10mu}\sum_{A\subset B_n,\atop {|A|=k, \atop |CC(A)|\leq k-1}}\mspace{-10mu} \Pr \left[\bigcap_{x\in A}\left\{ \pi_x \leq a_n \zeta \right\} \right]
\end{align*}
which converges to zero by virtue of Lemma \ref{lem:AuxConvInProb}.

Let us now consider the first term on the RHS of \eqref{equ:Cor1}.
Since for any $y\in B_n\backslash (A\cap \del A)$ the random variable $\pi_y$ is independent of $\left\{ \pi_x\right\}_{x\in A}$, the first sum on the RHS of \eqref{equ:Cor1} is
\begin{align}
  &\sum_{A\subset B_n,\atop |A|=k-1} \Pr \left[ \bigcap_{x\in A}\left\{ \pi_x \leq a_n \zeta \right\}\right] \Pr\left[ \min_{y\in B_n\backslash (A\cap \del A)} \pi_y > a_n \zeta \right]\nonumber\\
  &\mspace{250mu}\geq \Pr\left[ \min_{y\in B_n} \pi_y > a_n \zeta \right] \sum_{A\subset B_n,\atop |A|=k-1} \Pr \left[ \bigcap_{x\in A}\left\{ \pi_x \leq a_n \zeta \right\}\right]\label{equ:Cor3}\, .
\end{align}
Due to \eqref{equ:Cor0}, the first factor in the above RHS converges to $\exp\left( -\zeta^{2d\gamma} \right)$.
As a part of the proof for \eqref{equ:Cor0}, we have also shown that the second factor converges to $\zeta^{2(k-1)d\gamma}/(k-1)!$.
It thus remains to find an upper bound for the LHS of \eqref{equ:Cor3}.
Similar to the proof of \eqref{equ:Cor0}, we let $l$ be an even integer and estimate for all sequences of subsets $A_n\subset B_n$ with the constraint $|A_n|=k-1$ that
\begin{align*}
  \Pr\left[ \min_{y\in B_n\backslash (A_n\cap \del A_n)} \pi_y > a_n \zeta \right]
  &\leq 1 + \sum_{m=1}^{l}(-1)^{m}\sum_{M\subset B_n\backslash (A_n\cap \del A_n),\atop |M|=m} \Pr \left[ \bigcap_{x\in M}\left\{ \pi_{x} \leq a_n\zeta\right\} \right]\\
  &=1 + \sum_{m=1}^{l}(-1)^{m}\mspace{-5mu}\sum_{M\subset B_n\backslash (A_n\cap \del A_n),\atop {|M|=m,\atop CC(M)=m}}\mspace{-5mu} \Pr \left[ \bigcap_{x\in M}\left\{ \pi_{x} \leq a_n\zeta\right\} \right]\\
  &\mspace{100mu}+ \sum_{m=1}^{l}\sum_{q=1}^{m-1}\sum_{M\subset B_n\backslash (A_n\cap \del A_n),\atop {|M|=m,\atop CC(M)=q}} \Pr \left[ \bigcap_{x\in M}\left\{ \pi_{x} \leq a_n\zeta\right\} \right]
\end{align*}
According to Lemma \ref{lem:AuxConvInProb}, the supremum of the last sum on the above RHS taken over all sequences $A_n\subset B_n$ with $|A_n|=k-1$ converges to zero.
For the first sum we observe that since $|CC(M)|=|M|$, the set $M$ is sparse and therefore $\left\{ \pi_x \right\}_{x\in M}$ is a set of independent random variables.
It follows that
\begin{align*}
  \sum_{m=1}^{l}(-1)^{m}\mspace{-5mu}\sum_{M\subset B_n\backslash (A_n\cap \del A_n),\atop {|M|=m,\atop CC(M)=m}}\mspace{-5mu} \Pr \left[ \bigcap_{x\in M}\left\{ \pi_{x} \leq a_n\zeta\right\} \right]
  &= \sum_{m=1}^{l}(-1)^{m} C_{m,m}^{(n)}(A)\Pr \left[ \pi_{0} \leq a_n\zeta\right]^m\\
  &\mspace{-100mu}= \sum_{m=1}^{l}(-1)^{m} \left( |B_n|^m / m! + O\left( |B_n|^{m-1} \right) \right)\Pr \left[ \pi_{0} \leq a_n\zeta\right]^m
\end{align*}
by Remark \ref{rem:LeadOrderCmq}.
Taking the supremum over all sequences of subsets $A_n\subset B_n$ with the constraint $|A_n|=k-1$, this still converges to $\sum_{m=0}^l \zeta^{2d\gamma m}/m!$.
Since this holds for every $l\in 2\N$ and we already have the lower bound \eqref{equ:Cor3}, the claim follows.
\end{proof}

\appendix
\section{Appendix}
For better readability we have shifted a rather lengthy computation in the proof of Lemma \ref{lem:MuClose2} to this appendix.
We start by inserting the definition of the Laplacian, i.e.,
\begin{align*}
    &\sum_{x\in \mathscr{B}^{(n)}_{l+m}}\left( \Lw_{(l+m,n)} \phi_{l,m+1}^{(n)} (x) + \mu_{l,m+1}^{(n)} \phi_{l,m+1}^{(n)}(x) \right)^2\nonumber\\
    &\mspace{105mu}= \sum_{x\in \mathscr{B}^{(n)}_{l+m}}\left( \sum_{z\colon z\sim x} w_{xz} \left( \left( \phi_{l,m+1}^{(n)}\mathds{1}_{\mathscr{B}^{(n)}_{l+m}}\right) (z) - \phi_{l,m+1}^{(n)}(x)\right)+ \mu_{l,m+1}^{(n)} \phi_{l,m+1}^{(n)}(x) \right)^2
\end{align*}
Now we rearrange the terms in order to cancel $\mu_{l,m+1}^{(n)} \phi_{l,m+1}^{(n)}(x)$, i.e.,
\begin{align*}
    \text{LHS}&= \sum_{x\in \mathscr{B}^{(n)}_{l+m}}\left( \sum_{z\colon z\sim x} w_{xz} \left( \left(\phi_{l,m+1}^{(n)}\mathds{1}_{\Bln}\right)(z) - \phi_{l,m+1}^{(n)}(x)\right)\mspace{10mu}+\mspace{10mu} \mu_{l,m+1}^{(n)} \phi_{l,m+1}^{(n)}(x)\right.\\
    &\mspace{350mu}\left.-  \sum_{z\colon z\sim x} w_{xz} \left( \phi_{l,m+1}^{(n)}\mathds{1}_{\Bln\backslash \mathscr{B}^{(n)}_{l+m}}\right)(z)\right)^2\, ,
\end{align*}
where the first two terms cancel.
The last term simplifies to
\begin{align}
    \text{LHS}\mspace{10mu}&=\mspace{10mu}\sum_{x\in \mathscr{B}^{(n)}_{l+m}}\left( \sum_{z\in \Bln\backslash\mathscr{B}^{(n)}_{l+m}\colon z\sim x} \mspace{-20mu}w_{xz} \phi_{l,m+1}^{(n)}(z)\right)^2\nonumber\\
    &\leq\mspace{10mu} \max_{z\in \Bln\backslash\mathscr{B}^{(n)}_{l+m}} \left( \phi_{l,m+1}^{(n)}(z)\right)^2 \sum_{x\in B_n}\mspace{-3mu} \Biggl( \sum_{z\in \Bln\backslash\mathscr{B}^{(n)}_{l+m}\colon z\sim x} \mspace{-15mu}w_{xz} \Biggr)^2\, .
\label{equ:App1}
\end{align}

\paragraph{Acknowledgment.}
I am grateful to Wolfgang König for his very useful suggestions.

\bibliographystyle{alpha}

\end{document}